\documentclass[numreferences]{amsart}
\usepackage{latexsym}
\usepackage{graphicx}
\usepackage{amsmath}
\usepackage{mathrsfs}
\usepackage{psfrag,epstopdf}
\usepackage{cite}
\newtheorem{theorem}{Theorem}[section]
\newtheorem{lemma}[theorem]{Lemma}

\newtheorem{corollary}[theorem]{Corollary}
\newtheorem{definition}[theorem]{Definition}
\newenvironment{remark}{\noindent \textbf{Remark}.}{\hfill $\square$}

\newenvironment{proofof}[1]{\noindent \textit{Proof of #1.}}{\hfill $\square$}

\newcommand{\divg}{\mathop{\rm div}\nolimits}
\newcommand{\grad}{\mathop{\rm grad}\nolimits}

\newcommand{\dR}{\mathbb{R}}

\title[3D Lam\'{e} System]{Wiener Type Regularity of a Boundary Point for the 3D Lam\'{e} System}
\author{G. Luo \and V. G. Maz'ya}
\begin{document}
\maketitle


\section*{Abstract}
In this paper, we study the 3D Lam\'{e} system and establish its weighted positive definiteness for a certain range of elastic constants. By modifying the general theory developed in \cite{Maz'ya2002}, we then show, under the assumption
of weighted positive definiteness, that the divergence of the classical Wiener integral for a boundary point guarantees the continuity of solutions to the Lam\'{e} system at this point.

\section{The Main Results}
In our previous work \cite{Luo2007}, we studied weighted integral inequalities of the type
\begin{equation}
  \int_{\Omega} Lu \cdot \Psi u\,dx \geq 0
  \label{eqn_wpd}
\end{equation}
for general second order elliptic systems $L$ in $\dR^{n}\ (n \geq 3)$. For weights that are smooth and positive homogeneous of order $2-n$, we have shown that $L$ is positive definite in the sense of \eqref{eqn_wpd} only if the weight is the fundamental
matrix of $L$, possibly multiplied by a semi-positive definite constant matrix.

A question that arises naturally is under what conditions are elliptic systems indeed positive definite with such weights. Although it is difficult to answer this question in general, we study, as a special case, the 3D Lam\'{e} system
\begin{displaymath}
  L u = -\Delta u - \alpha \grad \divg u,\qquad u = (u_{1}, u_{2}, u_{3})^{T}
\end{displaymath}
in this paper, deriving sufficient conditions for its weighted positive definiteness and showing that some restrictions on the elastic constants are inevitable. By modifying the general theory developed in \cite{Maz'ya2002}, we then show that the
divergence of the classical Wiener integral for a boundary point guarantees the continuity of solutions to the Lam\'{e} system at this point, assuming the weighted positive definiteness.

We first recall the following definition.
\begin{definition}
Let $L$ be the 3D Lam\'{e} system
\begin{displaymath}
  L u = -\Delta u - \alpha \grad \divg u = -D_{kk} u_{i} - \alpha D_{ki} u_{k}\qquad (i = 1,2,3),
\end{displaymath}
where as usual repeated indices indicate summation. The system $L$ is said to be positive definite with weight $\Psi(x) = (\Psi_{ij}(x))_{i,j=1}^{3}$ if
\begin{equation}
  \int_{\dR^{3}} (Lu)^{T} \Psi u\,dx = -\int_{\dR^{3}} \bigl[ D_{kk} u_{i}(x) + \alpha D_{ki} u_{k}(x) \bigr] u_{j}(x) \Psi_{ij}(x)\,dx \geq 0
  \label{eqn_wpd1}
\end{equation}
for all real valued, smooth vector functions $u = (u_{i})_{i=1}^{3},\ u_{i} \in C_{0}^{\infty}(\dR^{3} \setminus \{0\})$. As usual, $D$ denotes the gradient $(D_{1},D_{2},D_{3})^{T}$ and $Du$ is the Jacobian matrix of $u$.
\end{definition}

\begin{remark}
The 3D Lam\'{e} system satisfies the strong elliptic condition if and only if $\alpha > -1$, and we will make this assumption throughout this paper.
\end{remark}

The fundamental matrix of the 3D Lam\'{e} system is given by $\Phi = (\Phi_{ij})_{i,j=1}^{3}$, where
\begin{align}
  \Phi_{ij} & = c_{\alpha} r^{-1} \Bigl( \delta_{ij} + \frac{\alpha}{\alpha+2} \omega_{i} \omega_{j} \Bigr)\qquad (i,j = 1,2,3), \label{eqn_Lame_fs} \\
  c_{\alpha} & = \frac{\alpha+2}{8\pi(\alpha+1)} > 0. \nonumber
\end{align}
As usual, $\delta_{ij}$ is the Kronecker delta, $r = |x|$ and $\omega_{i} = x_{i}/|x|$.

The first result we shall prove in this paper is the following
\begin{theorem}
The 3D Lam\'{e} system $L$ is positive definite with weight $\Phi$ when $\alpha_{-} < \alpha < \alpha_{+}$, where $\alpha_{-} \approx -0.194$ and $\alpha_{+} \approx 1.524$. It is not positive definite with weight $\Phi$ when $\alpha < \alpha_{-}^{(c)}
\approx -0.902$ or $\alpha > \alpha_{+}^{(c)} \approx 39.450$.
\label{thm_Lame_wpd}
\end{theorem}

The proof of this theorem is given in Section \ref{sec_Lame_wpd}.

Let $\Omega$ be an open set in $\dR^{3}$ and consider the Dirichlet problem
\begin{equation}
  Lu = f,\qquad f_{i} \in C_{0}^{\infty}(\Omega),\ u_{i} \in \mathring{H}^{1}(\Omega).
  \label{eqn_Dirichlet}
\end{equation}
As usual, $\mathring{H}^{1}(\Omega)$ is the completion of $C_{0}^{\infty}(\Omega)$ in the Sobolev norm:
\begin{displaymath}
  \|f\|_{H^{1}(\Omega)} = \Bigl[ \|f\|_{L^{2}(\Omega)}^{2} + \|Df\|_{L^{2}(\Omega)}^{2} \Bigr]^{1/2}.
\end{displaymath}

\begin{definition}
The point $P \in \partial \Omega$ is regular with respect to $L$ if, for any $f = (f_{i})_{i=1}^{3},\ f_{i} \in C_{0}^{\infty}(\Omega)$, the solution of \eqref{eqn_Dirichlet} satisfies
\begin{equation}
  \lim_{\Omega \ni x \to P} u_{i}(x) = 0\qquad (i=1,2,3).
  \label{eqn_regular}
\end{equation}
\end{definition}

\begin{definition}
The classical harmonic capacity of a compact set $K$ in $\dR^{3}$ is given by:
\begin{displaymath}
  \mathrm{cap}(K) = \inf \biggl\{ \int_{\dR^{3}} |Df(x)|^{2}\,dx: f \in \mathcal{A}(K) \biggr\},
\end{displaymath}
where
\begin{displaymath}
  \mathcal{A}(K) = \Bigl\{ f \in C_{0}^{\infty}(\dR^{3}): f = 1 \text{ in a neighborhood of $K$} \Bigr\}.
\end{displaymath}
Note that an equivalent definition of $\mathrm{cap}(K)$ can be obtained by replacing $\mathcal{A}(K)$ with $\mathcal{A}_{1}(K)$ where
\begin{displaymath}
  \mathcal{A}_{1}(K) = \Bigl\{ f \in C_{0}^{\infty}(\dR^{3}): f \geq 1 \text{ on $K$} \Bigr\}.
\end{displaymath}
(\cite{Maz'ya1985}, sec. 2.2.1).
\end{definition}

Using Theorem \ref{thm_Lame_wpd}, we will prove that the divergence of the classical Wiener integral for a boundary point $P$ guarantees its regularity with respect to the Lam\'{e} system. To simplify notations we assume, without loss of generality, that
$P = O$ is the origin of the space.

\begin{theorem}
Suppose the 3D Lam\'{e} system $L$ is positive definite with weight $\Phi$. Then $O \in \partial \Omega$ is regular with respect to $L$ if
\begin{equation}
  \int_{0}^{1} \mathrm{cap}(\bar{B}_{\rho} \setminus \Omega) \rho^{-2}\,d\rho = \infty.
  \label{eqn_Lame_wt}
\end{equation}
As usual, $B_{\rho}$ is the open ball centered at $O$ with radius $\rho$.
\label{thm_Lame_wt}
\end{theorem}

The proof of this theorem is given in Section \ref{sec_Lame_wt_suf}.

\section{Proof of Theorem \ref{thm_Lame_wpd}\label{sec_Lame_wpd}}
We start the proof of Theorem \ref{thm_Lame_wpd} by rewriting the integral
\begin{displaymath}
  \int_{\dR^{3}} (L u)^{T} \Phi u\,dx = -\int_{\dR^{3}} \bigl( D_{kk} u_{i} + \alpha D_{ki} u_{k} \bigr) u_{j} \Phi_{ij}\,dx
\end{displaymath}
into a more revealing form. In the following, we shall write $\int f\,dx$ instead of $\int_{\dR^{3}} f\,dx$, and by $u_{ii}^{2}$ we always mean $\sum_{i=1}^{3} u_{ii}^{2}$; to express $(\sum_{i=1}^{3} u_{ii})^{2}$ we will write $u_{ii} u_{jj}$ instead.
Furthermore, we always assume $u_{i} \in C_{0}^{\infty}(\dR^{3})$ unless otherwise stated.

\begin{lemma}
\begin{equation}
  \int (L u)^{T} \Phi u\,dx = \frac{1}{2} |u(0)|^{2} + \mathscr{B}(u,u)
  \label{eqn_Lame_int_1}
\end{equation}
where
\begin{align*}
  \mathscr{B}(u,u) & = \frac{\alpha}{2} \int \bigl( u_{j} D_{k} u_{k} - u_{k} D_{k} u_{j} \bigr) D_{i} \Phi_{ij}\,dx \\
  &\quad{} + \int \bigl( D_{k} u_{i} D_{k} u_{j} + \alpha D_{k} u_{k} D_{i} u_{j} \bigr) \Phi_{ij}\,dx.
\end{align*}
\label{lmm_Lame_wpd_1}
\end{lemma}

\begin{proof}
By definition,
\begin{displaymath}
  \int (L u)^{T} \Phi u\,dx = -\int D_{kk} u_{i} \cdot u_{j} \Phi_{ij}\,dx - \alpha \int D_{ki} u_{k} \cdot u_{j} \Phi_{ij}\,dx =: I_{1} + I_{2}.
\end{displaymath}
Since $\Phi$ is symmetric, we have $\Phi_{ij} = \Phi_{ji}$ and
\begin{align*}
  I_{1} & = -\int D_{kk} u_{i} \cdot u_{j} \Phi_{ij}\,dx \\
  & = -\frac{1}{2} \int \Bigl[ D_{kk} (u_{i} u_{j}) - 2 D_{k} u_{i} D_{k} u_{j} \Bigr] \Phi_{ij}\,dx \\
  & = -\frac{1}{2} \int u_{i} u_{j} D_{kk} \Phi_{ij}\,dx + \int D_{k} u_{i} D_{k} u_{j} \cdot \Phi_{ij}\,dx.
\end{align*}
On the other hand, $\Phi$ is the fundamental matrix of $L$, so we have
\begin{displaymath}
  -D_{kk} \Phi_{ij} - \alpha D_{ki} \Phi_{kj} = \delta_{ij} \delta(x),
\end{displaymath}
and
\begin{align*}
  \lefteqn{-\frac{1}{2} \int u_{i} u_{j} D_{kk} \Phi_{ij}\,dx = \frac{1}{2} \int u_{i} u_{j} \Bigl[ \delta_{ij} \delta(x) + \alpha D_{ki} \Phi_{kj} \Bigr]\,dx} \\
  &\qquad\qquad = \frac{1}{2} |u(0)|^{2} - \frac{\alpha}{2} \int \bigl( D_{i} u_{i} \cdot u_{j} + u_{i} D_{i} u_{j} \bigr) D_{k} \Phi_{kj}\,dx \\
  &\qquad\qquad = \frac{1}{2} |u(0)|^{2} - \frac{\alpha}{2} \int \bigl( D_{k} u_{k} \cdot u_{j} + u_{k} D_{k} u_{j} \bigr) D_{i} \Phi_{ij}\,dx.
\end{align*}
Now $I_{2}$ can be written as
\begin{displaymath}
  I_{2} = \alpha \int D_{k} u_{k} \bigl( D_{i} u_{j} \cdot \Phi_{ij} + u_{j} D_{i} \Phi_{ij} \bigr)\,dx,
\end{displaymath}
and the lemma follows by adding up the results.
\end{proof}

\begin{remark}
With $\Phi(x)$ replaced by $\Phi_{y}(x) := \Phi(x-y)$, we have
\begin{align*}
  \int (L u)^{T} \Phi_{y} u\,dx & = \int (L u_{y})^{T} \Phi u_{y}\,dx\qquad\qquad (u_{y}(x) = u(x+y)) \\
  & = \frac{1}{2} |u_{y}(0)|^{2} + \mathscr{B}(u_{y},u_{y}) =: \frac{1}{2} |u(y)|^{2} + \mathscr{B}_{y}(u,u),
\end{align*}
where
\begin{align*}
  \mathscr{B}_{y}(u,u) & = \frac{\alpha}{2} \int \bigl( u_{j} D_{k} u_{k} - u_{k} D_{k} u_{j} \bigr) D_{i} \Phi_{y,ij}\,dx \\
  &\quad{} + \int \bigl( D_{k} u_{i} D_{k} u_{j} + \alpha D_{k} u_{k} D_{i} u_{j} \bigr) \Phi_{y,ij}\,dx.
\end{align*}
\end{remark}

To proceed, we introduce the following decomposition for $C_{0}^{\infty}(\dR^{3})$ functions:
\begin{displaymath}
  f(x) = \bar{f}(r) + g(x),\qquad \bar{f} \in C_{0}^{\infty}[0,\infty),\ g \in C_{0}^{\infty}(\dR^{3}),
\end{displaymath}
where
\begin{displaymath}
  \bar{f}(r) = \frac{1}{4\pi} \int_{S^{2}} f(r\omega)\,d\sigma.
\end{displaymath}
Note that
\begin{displaymath}
  \int_{S^{2}} g(r\omega)\,d\sigma = 0,\qquad \forall r \geq 0,
\end{displaymath}
so we may think of $\bar{f}$ as the ``0-th order harmonics'' of the function $f$. We shall show below in Lemma \ref{lmm_Lame_wpd_2} that all 0-th order harmonics in \eqref{eqn_Lame_int_1} are canceled out, so it is possible to control $u$ by $Du$.

\begin{lemma}
With the decomposition
\begin{equation}
  u_{i}(x) = \bar{u}_{i}(r) + v_{i}(x)\qquad (i = 1,2,3)
  \label{eqn_decomp}
\end{equation}
where
\begin{displaymath}
  \begin{cases}
    \displaystyle \bar{u}_{i}(r) = \frac{1}{4\pi} \int_{S^{2}} u_{i}(r\omega)\,d\sigma \\
    \displaystyle \int_{S^{2}} v_{i}(r\omega)\,d\sigma = 0
  \end{cases}
  \qquad \forall r \geq 0\qquad (i = 1,2,3),
\end{displaymath}
we have
\begin{equation}
  \int (L u)^{T} \Phi u\,dx = \frac{1}{2} |u(0)|^{2} + \mathscr{B}^{*}(u,u)
  \label{eqn_Lame_int_2}
\end{equation}
where
\begin{align}
  \mathscr{B}^{*}(u,u) & = \frac{\alpha}{2} \int \bigl( v_{j} D_{k} v_{k} - v_{k} D_{k} v_{j} \bigr) D_{i} \Phi_{ij}\,dx \label{eqn_B_star} \\
  &\quad{} + \int \bigl( D_{k} u_{i} D_{k} u_{j} + \alpha D_{k} u_{k} D_{i} u_{j} \bigr) \Phi_{ij}\,dx. \nonumber
\end{align}
\label{lmm_Lame_wpd_2}
\end{lemma}

\begin{proof}
By Lemma \ref{lmm_Lame_wpd_1}, it is enough to show
\begin{displaymath}
  \int \bigl( u_{j} D_{k} u_{k} - u_{k} D_{k} u_{j} \bigr) D_{i} \Phi_{ij}\,dx = \int \bigl( v_{j} D_{k} v_{k} - v_{k} D_{k} v_{j} \bigr) D_{i} \Phi_{ij}\,dx.
\end{displaymath}
Since
\begin{align*}
  \lefteqn{\int \bigl( u_{j} D_{k} u_{k} - u_{k} D_{k} u_{j} \bigr) D_{i} \Phi_{ij}\,dx} \\
  &\quad = \int \bigl( \bar{u}_{j} D_{k} \bar{u}_{k} - \bar{u}_{k} D_{k} \bar{u}_{j} \bigr) D_{i} \Phi_{ij}\,dx + \int \bigl( \bar{u}_{j} D_{k} v_{k} - \bar{u}_{k} D_{k} v_{j} \bigr) D_{i} \Phi_{ij}\,dx \\
  &\quad\quad{} + \int \bigl( v_{j} D_{k} \bar{u}_{k} - v_{k} D_{k} \bar{u}_{j} \bigr) D_{i} \Phi_{ij}\,dx + \int \bigl( v_{j} D_{k} v_{k} - v_{k} D_{k} v_{j} \bigr) D_{i} \Phi_{ij}\,dx \\
  &\quad =: I_{1} + I_{2} + I_{3} + I_{4},
\end{align*}
it suffices to show $I_{1} = I_{2} = I_{3} = 0$. Now
\begin{align}
  D_{i} \Phi_{ij} & = D_{i} \biggl[ c_{\alpha} r^{-1} \Bigl( \delta_{ij} + \frac{\alpha}{\alpha+2} \omega_{i} \omega_{j} \Bigr) \biggr] \nonumber \\
  & = -c_{\alpha} r^{-2} \omega_{i} \delta_{ij} + \frac{c_{\alpha} \alpha}{\alpha+2} r^{-2} \Bigl[ -\omega_{i}^{2} \omega_{j} + (\delta_{ii} - \omega_{i}^{2}) \omega_{j} + (\delta_{ji} - \omega_{j} \omega_{i}) \omega_{i} \Bigr] \nonumber \\
  & = -c_{\alpha} r^{-2} \omega_{j} + \frac{c_{\alpha} \alpha}{\alpha+2} r^{-2} \omega_{j} =: d_{\alpha} r^{-2} \omega_{j},
  \label{eqn_D_Phi}
\end{align}
where
\begin{displaymath}
  d_{\alpha} = \frac{-2c_{\alpha}}{\alpha+2} = \frac{-1}{4\pi(\alpha+1)}.
\end{displaymath}
We have
\begin{align*}
  I_{1} & = d_{\alpha} \int r^{-2} \omega_{j} \bigl( \bar{u}_{j} D_{r} \bar{u}_{k} \cdot \omega_{k} - \bar{u}_{k} D_{r} \bar{u}_{j} \cdot \omega_{k} \bigr)\,dx & & (D_r = \partial/\partial r) \\
  & = d_{\alpha} \int r^{-2} \bigl( \bar{u}_{j} D_{r} \bar{u}_{k} \cdot \omega_{j} \omega_{k} - \bar{u}_{k} D_{r} \bar{u}_{j} \cdot \omega_{k} \omega_{j} \bigr)\,dx = 0, \\
  I_{3} & = d_{\alpha} \int r^{-2} \bigl( v_{j} D_{r} \bar{u}_{k} \cdot \omega_{j} \omega_{k} - v_{k} D_{r} \bar{u}_{j} \cdot \omega_{k} \omega_{j} \bigr)\,dx = 0.
\end{align*}
As for $I_{2}$, we obtain
\begin{align*}
  I_{2} & = d_{\alpha} \int r^{-2} \bigl( \bar{u}_{j} D_{k} v_{k} \cdot \omega_{j} - \bar{u}_{k} D_{k} v_{j} \cdot \omega_{j} \bigr)\,dx \\
  & = d_{\alpha} \int r^{-2} \bigl( \bar{u}_{j} D_{k} v_{k} \cdot \omega_{j} - \bar{u}_{j} D_{j} v_{k} \cdot \omega_{k} \bigr)\,dx \\
  & = -\lim_{\epsilon \to 0^{+}} d_{\alpha} \int_{S^{2}} \Bigl[ \bar{u}_{j}(\epsilon) v_{k}(\epsilon \omega) \omega_{j} \omega_{k} - \bar{u}_{j}(\epsilon) v_{k}(\epsilon \omega) \omega_{j} \omega_{k} \Bigr]\,d\sigma \\
  &\quad{} - \lim_{\epsilon \to 0^{+}} d_{\alpha} \int_{\dR^{3} \setminus B_{\epsilon}} \biggl\{ v_{k} r^{-3} \Bigl[ -2 \bar{u}_{j} \omega_{j} \omega_{k} + r D_{r} \bar{u}_{j} \cdot \omega_{j} \omega_{k} + \bar{u}_{j} (\delta_{jk} - \omega_{j} \omega_{k}) \Bigr] \\
  &\qquad\qquad{} - v_{k} r^{-3} \Bigl[ -2 \bar{u}_{j} \omega_{j} \omega_{k} + r D_{r} \bar{u}_{j} \cdot \omega_{j} \omega_{k} + \bar{u}_{j} (\delta_{kj} - \omega_{k} \omega_{j}) \Bigr] \biggr\}\,dx = 0.
\end{align*}
The result follows.
\end{proof}

\begin{remark}
With $\Phi(x)$ replaced by $\Phi_{y}(x) := \Phi(x-y)$ and \eqref{eqn_decomp} replaced by
\begin{displaymath}
  u_{i}(x) = \bar{u}_{i}(r_{y}) + v_{i}(x)\qquad (i = 1,2,3),
\end{displaymath}
where $r_{y} = |x-y|$ and
\begin{displaymath}
  \begin{cases}
    \displaystyle \bar{u}_{i}(r_{y}) = \frac{1}{4\pi} \int_{S^{2}} u_{i}(y+r_{y}\omega)\,d\sigma \\
    \displaystyle \int_{S^{2}} v_{i}(y+r_{y}\omega)\,d\sigma = 0
  \end{cases}
  \qquad \forall r_{y} \geq 0\qquad (i = 1,2,3),
\end{displaymath}
we have
\begin{displaymath}
  \int (L u)^{T} \Phi_{y} u\,dx = \frac{1}{2} |u(y)|^{2} + \mathscr{B}^{*}_{y}(u,u)
\end{displaymath}
where
\begin{align*}
  \mathscr{B}^{*}_{y}(u,u) & = \frac{\alpha}{2} \int \bigl( v_{j} D_{k} v_{k} - v_{k} D_{k} v_{j} \bigr) D_{i} \Phi_{y,ij}\,dx \\
  &\quad{} + \int \bigl( D_{k} u_{i} D_{k} u_{j} + \alpha D_{k} u_{k} D_{i} u_{j} \bigr) \Phi_{y,ij}\,dx.
\end{align*}
\end{remark}

In the next Lemma, we use the definition of $\Phi$ and derive an explicit expression for the bilinear form $\mathscr{B}^{*}(u,u)$ defined in \eqref{eqn_B_star}.

\begin{lemma}
\begin{align}
  \lefteqn{\mathscr{B}^{*}(u,u) = c_{\alpha} \int \bigg\{ \frac{\alpha}{\alpha+2} r^{-2} \Bigl[ v_{k} (D_{k} v) \cdot \omega - (\divg v) (v \cdot \omega) \Bigr]} \label{eqn_Lame_int_3} \\
  &\hspace{0.4in}{} + r^{-1} \Bigl[ |D_{r} \bar{u}|^{2} + \alpha \frac{2\alpha+3}{\alpha+2} (D_{r} \bar{u}_{i})^{2} \omega_{i}^{2} + |D v|^{2} + \alpha (\divg v)^{2} \nonumber \\
  &\hspace{0.85in}{} + \frac{\alpha}{\alpha+2} |(D_{k} v) \cdot \omega|^{2} + \frac{\alpha^{2}}{\alpha+2} (\divg v) [\omega_{i} (D_{i} v) \cdot \omega] \nonumber \\
  &\hspace{0.85in}{} + \alpha \frac{3\alpha+4}{\alpha+2} (D_{r} \bar{u} \cdot \omega) (\divg v) + \alpha (D_{r} \bar{u} \cdot \omega) [\omega_{i} (D_{i} v) \cdot \omega] \Bigr] \biggr\}\,dx. \nonumber
\end{align}
\label{lmm_Lame_wpd_3}
\end{lemma}

Before proving this lemma, we need a simple yet important observation that will be useful in the following computation.

\begin{lemma}
Let $g \in C_{0}^{\infty}(\dR^{3})$ be such that
\begin{displaymath}
  \int_{S^{2}} g(r\omega)\,d\sigma = 0,\qquad \forall r \geq 0.
\end{displaymath}
Then
\begin{displaymath}
  \begin{cases}
    \displaystyle \int f(r) g(x)\,dx = 0 \\
    \displaystyle \int r^{-1} D f(r) \cdot D g(x)\,dx = 0
  \end{cases}
  \qquad \forall f \in C_{0}^{\infty}[0,\infty).
\end{displaymath}
\label{lmm_ortho}
\end{lemma}

\begin{proof}
By switching to polar coordinates, we easily see that
\begin{displaymath}
  \int f(r) g(x)\,dx = \int_{0}^{\infty} r^{2} f(r)\,dr \int_{S^{2}} g(r\omega)\,d\sigma = 0.
\end{displaymath}
On the other hand,
\begin{align*}
  \lefteqn{\int r^{-1} D f(r) \cdot D g(x)\,dx = \int r^{-1} D_{r} f D_{i} g \cdot \omega_{i}\,dx} \\
  &\qquad = -\int g \Bigl[ -r^{-2} (D_{r} f) \omega_{i}^{2} + r^{-1} (D_{rr} f) \omega_{i}^{2} + r^{-2} D_{r} f (\delta_{ii} - \omega_{i}^{2}) \Bigr]\,dx \\
  &\qquad = -\int g \bigl( r^{-2} D_{r} f + r^{-1} D_{rr} f \bigr)\,dx = 0,
\end{align*}
where the last equality follows by switching to polar coordinates.
\end{proof}

\begin{proofof}{Lemma \ref{lmm_Lame_wpd_3}}
By definition,
\begin{align*}
  \mathscr{B}^{*}(u,u) & = \frac{\alpha}{2} \int \bigl( v_{j} D_{k} v_{k} - v_{k} D_{k} v_{j} \bigr) D_{i} \Phi_{ij}\,dx \\
  &\quad{} + \int \bigl( D_{k} u_{i} D_{k} u_{j} + \alpha D_{k} u_{k} D_{i} u_{j} \bigr) \Phi_{ij}\,dx =: I_{1} + I_{2}.
\end{align*}
We have shown in Lemma \ref{lmm_Lame_wpd_2} that (see \eqref{eqn_D_Phi})
\begin{align*}
  I_{1} & = 2^{-1} \alpha d_{\alpha} \int r^{-2} \omega_{j} \bigl( v_{j} D_{k} v_{k} - v_{k} D_{k} v_{j} \bigr)\,dx \\
  & = \frac{c_{\alpha} \alpha}{\alpha+2} \int r^{-2} \Bigl[ v_{k} (D_{k} v) \cdot \omega - (\divg v) (v \cdot \omega) \Bigr]\,dx.
\end{align*}
On the other hand,
\begin{align*}
  I_{2} & = c_{\alpha} \int r^{-1} D_{k} u_{i} D_{k} u_{i}\,dx + \frac{c_{\alpha} \alpha}{\alpha+2} \int r^{-1} D_{k} u_{i} D_{k} u_{j} \cdot \omega_{i} \omega_{j}\,dx \\
  &\quad{} + c_{\alpha} \alpha \int r^{-1} D_{k} u_{k} D_{i} u_{i}\,dx + \frac{c_{\alpha} \alpha^{2}}{\alpha+2} \int r^{-1} D_{k} u_{k} D_{i} u_{j}\cdot \omega_{i} \omega_{j}\,dx \\
  & =: I_{3} + I_{4} + I_{5} + I_{6}.
\end{align*}
Substituting $u_{i} = \bar{u}_{i} + v_{i}$ into $I_{3}$ and using Lemma \ref{lmm_ortho} yields
\begin{align}
  I_{3} & = c_{\alpha} \int r^{-1} \bigl( D_{r} \bar{u}_{i} D_{r} \bar{u}_{i} \cdot \omega_{k}^{2} + D_{k} v_{i} D_{k} v_{i} \bigr)\,dx + 2c_{\alpha} \int r^{-1} D_{k} \bar{u}_{i} D_{k} v_{i}\,dx \nonumber \\
  & = c_{\alpha} \int r^{-1} \bigl( |D_{r} \bar{u}|^{2} + |D v|^{2} \bigr)\,dx.
  \label{eqn_Du2}
\end{align}
Next,
\begin{align*}
  I_{5} & = c_{\alpha} \alpha \int r^{-1} \bigl( D_{r} \bar{u}_{k} D_{r} \bar{u}_{i} \cdot \omega_{k} \omega_{i} + 2 D_{i} v_{i} D_{r} \bar{u}_{k} \cdot \omega_{k} + D_{k} v_{k} D_{i} v_{i} \bigr)\,dx.
\end{align*}
Note that for $k \neq i$,
\begin{displaymath}
  \int r^{-1} D_{r} \bar{u}_{k} D_{r} \bar{u}_{i} \cdot \omega_{k} \omega_{i}\,dx = \int_{0}^{\infty} r D_{r} \bar{u}_{k} D_{r} \bar{u}_{i}\,dr \int_{S^{2}} \omega_{k} \omega_{i}\,d\sigma = 0,
\end{displaymath}
and therefore
\begin{displaymath}
  I_{5} = c_{\alpha} \alpha \int r^{-1} \Bigl[ (D_{r} \bar{u}_{i})^{2} \omega_{i}^{2} + 2 (\divg v) (D_{r} \bar{u} \cdot \omega) + (\divg v)^{2} \Bigr]\,dx.
\end{displaymath}
As for $I_{4}$,
\begin{align*}
  I_{4} & = \frac{c_{\alpha} \alpha}{\alpha+2} \int r^{-1} D_{k} (\bar{u}_{i} + v_{i}) D_{k} (\bar{u}_{j} + v_{j}) \cdot \omega_{i} \omega_{j}\,dx \\
  & = \frac{c_{\alpha} \alpha}{\alpha+2} \int r^{-1} \bigl( D_{r} \bar{u}_{i} D_{r} \bar{u}_{j} \cdot \omega_{i} \omega_{j} \omega_{k}^{2} + D_{r} \bar{u}_{i} D_{k} v_{j} \cdot \omega_{i} \omega_{j} \omega_{k} \\
  &\hspace{1.05in}{} + D_{k} v_{i} D_{r} \bar{u}_{j} \cdot \omega_{i} \omega_{j} \omega_{k} + D_{k} v_{i} D_{k} v_{j} \cdot \omega_{i} \omega_{j} \bigr)\,dx \\
  & = \frac{c_{\alpha} \alpha}{\alpha+2} \int r^{-1} \Bigl[ (D_{r} \bar{u}_{i})^{2} \omega_{i}^{2} + 2 (D_{r} \bar{u} \cdot \omega) [\omega_{k} (D_{k} v) \cdot \omega] + |D_{k} v \cdot \omega|^{2} \Bigr]\,dx.
\end{align*}
Similarly,
\begin{align*}
  I_{6} & = \frac{c_{\alpha} \alpha^{2}}{\alpha+2} \int r^{-1} D_{k} (\bar{u}_{k} + v_{k}) D_{i} (\bar{u}_{j} + v_{j}) \cdot \omega_{i} \omega_{j}\,dx \\
  & = \frac{c_{\alpha} \alpha^{2}}{\alpha+2} \int r^{-1} \bigl( D_{r} \bar{u}_{k} D_{r} \bar{u}_{j} \cdot \omega_{i}^{2} \omega_{j} \omega_{k} + D_{r} \bar{u}_{k} D_{i} v_{j} \cdot \omega_{i} \omega_{j} \omega_{k} \\
  &\hspace{1.05in}{} + D_{r} \bar{u}_{j} D_{k} v_{k} \cdot \omega_{i}^{2} \omega_{j} + D_{k} v_{k} D_{i} v_{j} \cdot \omega_{i} \omega_{j} \bigr)\,dx \\
  & = \frac{c_{\alpha} \alpha^{2}}{\alpha+2} \int r^{-1} \Bigl[ (D_{r} \bar{u}_{j})^{2} \omega_{j}^{2} + (D_{r} \bar{u} \cdot \omega) [\omega_{i} (D_{i} v) \cdot \omega] \\
  &\hspace{1.05in}{} + (D_{r} \bar{u} \cdot \omega) (\divg v) + (\divg v) [\omega_{i} (D_{i} v) \cdot \omega] \Bigr]\,dx.
\end{align*}
The lemma follows by adding up all these integrals.
\end{proofof}

With the help of Lemma \ref{lmm_Lame_wpd_3}, we now complete the proof of Theorem \ref{thm_Lame_wpd}.

\begin{proofof}{Theorem \ref{thm_Lame_wpd}}
By Lemma \ref{lmm_Lame_wpd_2} and \ref{lmm_Lame_wpd_3},
\begin{displaymath}
  -c_{\alpha}^{-1} \int (L u)^{T} \Phi u\,dx = \frac{1}{2} c_{\alpha}^{-1} |u(0)|^{2} + I_{1} + I_{2} + I_{3},
\end{displaymath}
where
\begin{align*}
  I_{1} & = \int r^{-1} \biggl[ |D_{r} \bar{u}|^{2} + \alpha \frac{2\alpha+3}{\alpha+2} (D_{r} \bar{u}_{i})^{2} \omega_{i}^{2} + |D v|^{2} \\
  &\hspace{0.65in}{} + \alpha (\divg v)^{2} + \frac{\alpha}{\alpha+2} |(D_{k} v) \cdot \omega|^{2} \biggr]\,dx, \\
  \intertext{} \\
  I_{2} & = \int r^{-1} \biggl[ \frac{\alpha^{2}}{\alpha+2} (\divg v) [\omega_{i} (D_{i} v) \cdot \omega] + \alpha \frac{3\alpha+4}{\alpha+2} (D_{r} \bar{u} \cdot \omega) (\divg v) \\
  &\hspace{0.65in}{} + \alpha (D_{r} \bar{u} \cdot \omega) [\omega_{i} (D_{i} v) \cdot \omega] \biggr]\,dx, \\
  I_{3} & = \int \frac{\alpha}{\alpha+2} r^{-2} \Bigl[ v_{k} (D_{k} v) \cdot \omega - (\divg v) (v \cdot \omega) \Bigr]\,dx.
\end{align*}
Consider first the case $\alpha \geq 0$. By switching to polar coordinates, we have
\begin{align*}
  I_{1} & \geq \int r^{-1} \biggl[ |D_{r} \bar{u}|^{2} + \alpha \frac{2\alpha+3}{\alpha+2} (D_{r} \bar{u}_{i})^{2} \omega_{i}^{2} + |D v|^{2} + \alpha (\divg v)^{2} \biggr]\,dx \\
  & = \int_{0}^{\infty} r \biggl[ \Bigl( 1 + \frac{\alpha}{3} \cdot \frac{2\alpha+3}{\alpha+2} \Bigr) \|D_{r} \bar{u}\|_{\omega}^{2} + \|D v\|_{\omega}^{2} + \alpha \|\divg v\|_{\omega}^{2} \biggr]\,dr,
\end{align*}
where we have written $\|\cdot\|_{\omega}$ for $\|\cdot\|_{L^{2}(S^{2})}$ and used the fact that
\begin{displaymath}
  \int_{S^{2}} (D_{r} \bar{u}_{i})^{2} \omega_{i}^{2}\,d\sigma = \frac{4\pi}{3} \sum_{i=1}^{3} (D_{r} \bar{u}_{i})^{2} = \frac{1}{3} \int_{S^{2}} |D_{r} \bar{u}|^{2}\,d\sigma = \frac{1}{3} \|D_{r} \bar{u}\|_{\omega}^{2}.
\end{displaymath}
Next,
\begin{align*}
  |I_{2}| & \leq \int r^{-1} \biggl[ \frac{\alpha^{2}}{\alpha+2} |\divg v| |D v| + \alpha \frac{3\alpha+4}{\alpha+2} |D_{r} \bar{u} \cdot \omega| |\divg v| + \alpha |D_{r} \bar{u} \cdot \omega| |D v| \biggr]\,dx \\
  & \leq \int_{0}^{\infty} r \biggl[ \frac{\alpha^{2}}{\alpha+2} \|\divg v\|_{\omega} \|D v\|_{\omega} + \frac{\alpha}{\sqrt{3}} \cdot \frac{3\alpha+4}{\alpha+2} \|D_{r} \bar{u}\|_{\omega} \|\divg v\|_{\omega} \\
  &\hspace{0.6in}{} + \frac{\alpha}{\sqrt{3}} \|D_{r} \bar{u}\|_{\omega} \|D v\|_{\omega} \biggr]\,dr,
\end{align*}
where we have used
\begin{align*}
  \|D_{r} \bar{u} \cdot \omega\|_{\omega}^{2} & = \int_{S^{2}} D_{r} \bar{u}_{i} D_{r} \bar{u}_{j} \cdot \omega_{i} \omega_{j}\,d\sigma \\
  & = D_{r} \bar{u}_{i} D_{r} \bar{u}_{j} \cdot \frac{4\pi}{3} \delta_{ij} = \frac{4\pi}{3} \sum_{i=1}^{3} (D_{r} \bar{u}_{i})^{2} = \frac{1}{3} \|D_{r} \bar{u}\|_{\omega}^{2}.
\end{align*}
As for $I_{3}$, we note that
\begin{align*}
  |I_{3}| & \leq \frac{\alpha}{\alpha+2} \int r^{-2} \bigl( |v| |D v| + |v| |\divg v| \bigr)\,dx \\
  & \leq \frac{\alpha}{\alpha+2} \int_{0}^{\infty} \|v\|_{\omega} \bigl( \|D v\|_{\omega} + \|\divg v\|_{\omega} \bigr)\,dr.
\end{align*}
Since 2 is the first non-trivial eigenvalue of the Laplace-Beltrami operator on $S^{2}$, we have
\begin{align}
  \|v\|_{\omega}^{2} & = \int_{S^{2}} |v(r\omega)|^{2}\,d\sigma \leq \frac{1}{2} \int_{S^{2}} |D_{\omega}[v(r\omega)]|^{2}\,d\sigma \nonumber \\
  & = \frac{r^{2}}{2} \int_{S^{2}} |(D_{\omega}v)(r\omega)|^{2}\,d\sigma \leq \frac{r^{2}}{2} \|D v\|_{\omega}^{2}, \label{eqn_Hardy_crit}
\end{align}
where $D_{\omega}$ is the gradient operator on $S^{2}$. Thus
\begin{displaymath}
  |I_{3}| \leq \frac{1}{\sqrt{2}} \cdot \frac{\alpha}{\alpha+2} \int_{0}^{\infty} r \Bigl[ \|D v\|_{\omega}^{2} + \|D v\|_{\omega} \|\divg v\|_{\omega} \Bigr]\,dr,
\end{displaymath}
and by putting all pieces together we obtain
\begin{equation}
  I_{1}+I_{2}+I_{3} \geq \int_{0}^{\infty} r \bigl( w^{T} B_{+} w \bigr)\,dr,
  \label{eqn_Lame_est_alpha_+}
\end{equation}
where
\begin{align*}
  w & = \bigl( \|D_{r} \bar{u}\|_{\omega}, \|D v\|_{\omega}, \|\divg v\|_{\omega} \bigr)^{T}, \\
  B_{+} & =
  \begin{bmatrix}
    1 + \dfrac{\alpha}{3} \cdot \dfrac{2\alpha+3}{\alpha+2} & -\dfrac{\alpha}{2\sqrt{3}} & -\dfrac{\alpha}{2\sqrt{3}} \cdot \dfrac{3\alpha+4}{\alpha+2} \\
    -\dfrac{\alpha}{2\sqrt{3}} & 1-\dfrac{1}{\sqrt{2}} \cdot \dfrac{\alpha}{\alpha+2} & -\dfrac{\alpha}{2} \cdot \dfrac{\alpha+2^{-1/2}}{\alpha+2} \\
    -\dfrac{\alpha}{2\sqrt{3}} \cdot \dfrac{3\alpha+4}{\alpha+2} & -\dfrac{\alpha}{2} \cdot \dfrac{\alpha+2^{-1/2}}{\alpha+2} & \alpha
  \end{bmatrix}.
\end{align*}
Clearly, the weighted positive definiteness of $L$ follows from the positive definiteness of $B_{+}$, because the latter implies, for some $c > 0$, that
\begin{align*}
  \int_{0}^{\infty} r \bigl( w^{T} B_{+} w \bigr)\,dr & \geq c \int_{0}^{\infty} r |w|^{2}\,dr \\
  & \geq c \int_{0}^{\infty} r \bigl( \|D_{r} \bar{u}\|_{\omega}^{2} + \|Dv\|_{\omega}^{2} \bigr)\,dr = c \int r^{-1} |Du|^{2}\,dx.
\end{align*}
The positive definiteness of $B_{+}$, on the other hand, is equivalent to the positivity of the determinants of all leading principal minors of $B_{+}$:
\begin{subequations}\label{eqn_Lame_wpd_cond_alpha_+}
\begin{align}
  p_{+,1}(\alpha) & = \frac{2\alpha^{2}+6\alpha+6}{3(\alpha+2)} > 0, \label{eqn_Lame_wpd_cond_alpha_+_p1} \\
  p_{+,2}(\alpha) & = -\frac{1}{12(\alpha+2)^{2}} \Bigl[ \alpha^{4} - 4(1-\sqrt{2}) \alpha^{3} - 12(3-\sqrt{2}) \alpha^{2} \nonumber \\
  &\hspace{1.1in}{} - 12(6-\sqrt{2}) \alpha - 48 \Bigr] > 0, \label{eqn_Lame_wpd_cond_alpha_+_p2} \\
  \intertext{}
  p_{+,3}(\alpha) & = -\frac{\alpha}{12(\alpha+2)^{3}} \Bigl[ 6\alpha^{5} + (23+3\sqrt{2}) \alpha^{4} + (13+19\sqrt{2}) \alpha^{3} \nonumber \\
  &\hspace{1.1in}{} - (77-38\sqrt{2}) \alpha^{2} - (157-24\sqrt{2}) \alpha - 96 \Bigr] > 0. \label{eqn_Lame_wpd_cond_alpha_+_p3}
\end{align}
\end{subequations}
With the help of computer algebra packages, we find that \eqref{eqn_Lame_wpd_cond_alpha_+} holds for $0 \leq \alpha < \alpha_{+}$, where $\alpha_{+} \approx 1.524$ is the largest real root of $p_{+,3}$.

The estimates of $I_{1},\ I_{2}$, and $I_{3}$ are slightly different when $\alpha < 0$, since now the quadratic term $\alpha \|\divg v\|_{\omega}^{2}$ in $I_{1}$ is negative. This means that it is no longer possible to control the $\|\divg v\|_{\omega}$
terms in $I_{2},\ I_{3}$ by $\alpha \|\divg v\|_{\omega}^{2}$, and in order to obtain positivity we need to bound $\|\divg v\|_{\omega}$ by $\|D v\|_{\omega}$ as follows:
\begin{displaymath}
  \|\divg v\|_{\omega}^{2} \leq 3 \|D v\|_{\omega}^{2}.
\end{displaymath}
This leads to the following revised estimates:
\begin{align*}
  I_{1} & \geq \int_{0}^{\infty} r \biggl[ \Bigl( 1 + \frac{\alpha}{3} \cdot \frac{2\alpha+3}{\alpha+2} \Bigr) \|D_{r} \bar{u}\|_{\omega}^{2} + \|D v\|_{\omega}^{2} + 3\alpha \|D v\|_{\omega}^{2} + \frac{\alpha}{\alpha+2} \|D v\|_{\omega}^{2} \biggr]\,dr, \\
  |I_{2}| & \leq \int_{0}^{\infty} r \biggl[ \frac{\sqrt{3}\,\alpha^{2}}{\alpha+2} \|D v\|_{\omega}^{2} - \alpha \frac{3\alpha+4}{\alpha+2} \|D_{r} \bar{u}\|_{\omega} \|D v\|_{\omega} - \frac{\alpha}{\sqrt{3}} \|D_{r} \bar{u}\|_{\omega} \|D v\|_{\omega} \biggr]\,dr, \\
  |I_{3}| & \leq -\frac{1}{\sqrt{2}} \cdot \frac{\alpha}{\alpha+2} \int_{0}^{\infty} r \Bigl[ \|D v\|_{\omega}^{2} + \sqrt{3}\,\|D v\|_{\omega}^{2} \Bigr]\,dr.
\end{align*}
Hence
\begin{equation}
  I_{1}+I_{2}+I_{3} \geq \int_{0}^{\infty} r \bigl( w^{T} B_{-} w \bigr)\,dr,
  \label{eqn_Lame_est_alpha_-}
\end{equation}
where
\begin{align*}
  w & = \bigl( \|D_{r} \bar{u}\|_{\omega}, \|D v\|_{\omega} \bigr)^{T}, \\
  B_{-} & =
  \begin{bmatrix}
    1 + \dfrac{\alpha}{3} \cdot \dfrac{2\alpha+3}{\alpha+2} & \dfrac{\alpha}{2} \cdot \dfrac{3\alpha+4}{\alpha+2} + \dfrac{\alpha}{2\sqrt{3}} \\
    \dfrac{\alpha}{2} \cdot \dfrac{3\alpha+4}{\alpha+2} + \dfrac{\alpha}{2\sqrt{3}} & 1 + 3\alpha + \dfrac{\alpha}{\alpha+2} \Bigl( 1 + \dfrac{1+\sqrt{3}}{\sqrt{2}} - \sqrt{3}\,\alpha \Bigr)
  \end{bmatrix}.
\end{align*}
The positive definiteness of $B_{-}$ is equivalent to:
\begin{subequations}\label{eqn_Lame_wpd_cond_alpha_-}
\begin{align}
  p_{-,1}(\alpha) & = \frac{2\alpha^{2}+6\alpha+6}{3(\alpha+2)} > 0, \label{eqn_Lame_wpd_cond_alpha_-_p1} \\
  p_{-,2}(\alpha) & = \frac{1}{6(\alpha+2)^{2}} \Bigl[ -(2+7\sqrt{3}) \alpha^{4} + 2(15+\sqrt{2}-11\sqrt{3}+\sqrt{6}) \alpha^{3} \nonumber \\
  &\quad{} + 2(57+3\sqrt{2}-10\sqrt{3}+3\sqrt{6}) \alpha^{2} + 6(20+\sqrt{2}+\sqrt{6}) \alpha + 24 \Bigr] > 0, \label{eqn_Lame_wpd_cond_alpha_-_p2}
\end{align}
\end{subequations}
and \eqref{eqn_Lame_wpd_cond_alpha_-} holds for $\alpha_{-} < \alpha < 0$, where $\alpha_{-} \approx -0.194$ is the smallest real root of $p_{-,2}$.

Now we show that the 3D Lam\'{e} system is not positive definite with weight $\Phi$ when $\alpha$ is either too close to $-1$ or too large. By Proposition 3.11 in \cite{Luo2007}, the 3D Lam\'{e} system is positive definite with weight $\Phi$ only if
\begin{displaymath}
  \sum_{i,\beta,\gamma} A_{ip}^{\beta \gamma} \xi_{\beta} \xi_{\gamma} \Phi_{ip}(\omega) \geq 0,\qquad \forall \xi \in \dR^{3},\ \forall \omega \in S^{2}\qquad (p = 1,2,3),
\end{displaymath}
where
\begin{displaymath}
  A_{ij}^{\beta \gamma} = \delta_{ij} \delta_{\beta \gamma} + \frac{\alpha}{2} (\delta_{i\beta} \delta_{j\gamma} + \delta_{i\gamma} \delta_{j\beta})
\end{displaymath}
and (see equation \eqref{eqn_Lame_fs})
\begin{displaymath}
  \Phi_{ij}(\omega) = c_{\alpha} r^{-1} \Bigl( \delta_{ij} + \frac{\alpha}{\alpha+2} \omega_{i} \omega_{j} \Bigr)\qquad (i,j = 1,2,3).
\end{displaymath}
This means, in particular, that the matrix
\begin{align*}
  A(\omega; \alpha) & := \biggl( \sum_{i=1}^{3} A_{i1}^{\beta \gamma} \Phi_{i1}(\omega) \biggr)_{\beta,\gamma=1}^{3} \\
  & = \frac{c_{\alpha} r^{-1}}{2(\alpha+2)}
  \begin{bmatrix}
    2(\alpha+1) (\alpha+2 + \alpha \omega_{1}^{2}) & \alpha^{2} \omega_{1} \omega_{2} & \alpha^{2} \omega_{1} \omega_{3} \\
    \alpha^{2} \omega_{1} \omega_{2} & 2(\alpha+2 + \alpha \omega_{1}^{2}) & 0 \\
    \alpha^{2} \omega_{1} \omega_{3} & 0 & 2(\alpha+2 + \alpha \omega_{1}^{2})
  \end{bmatrix}
\end{align*}
is semi-positive definite for any $\omega \in S^{2}$ if the 3D Lam\'{e} system is positive definite with weight $\Phi$. But $A(\omega; \alpha)$ is semi-positive definite only if the determinant of its leading principal minor
\begin{align*}
  d_{2}(\omega; \alpha) & := \det
  \begin{bmatrix}
    2(\alpha+1) (\alpha+2 + \alpha \omega_{1}^{2}) & \alpha^{2} \omega_{1} \omega_{2} \\
    \alpha^{2} \omega_{1} \omega_{2} & 2(\alpha+2 + \alpha \omega_{1}^{2})
  \end{bmatrix} \\
  & = 4(\alpha+1) (\alpha+2 + \alpha \omega_{1}^{2})^{2} - \alpha^{4} \omega_{1}^{2} \omega_{2}^{2}
\end{align*}
is non-negative, and elementary estimate shows that
\begin{align*}
  \min_{\omega \in S^{2}} d_{2}(\omega; \alpha) & \leq d_{2} \bigl[ (2^{-1/2},2^{-1/2},0); \alpha \bigr] \\
  & = (\alpha+1) (3\alpha+4)^{2} - \frac{\alpha^{4}}{4} =: q(\alpha).
\end{align*}
It follows that the 3D Lam\'{e} system is not positive definite with weight $\Phi$ when $q(\alpha) < 0$, which holds for $\alpha < \alpha_{-}^{(c)} \approx -0.902$ or $\alpha > \alpha_{+}^{(c)} \approx 39.450$.
\end{proofof}

\begin{remark}
We have in fact shown that, for $\alpha_{-} < \alpha < \alpha_{+}$ and some $c > 0$ depending on $\alpha$,
\begin{displaymath}
  \int (L u)^{T} \Phi u\,dx \geq \frac{1}{2} |u(0)|^{2} + c \int |Du(x)|^{2} \frac{dx}{|x|}.
\end{displaymath}
If we replace $\Phi(x)$ by $\Phi_{y}(x) := \Phi(x-y)$, then
\begin{align}
  \int (L u)^{T} \Phi_{y} u\,dx & = \int [L u(x+y)]^{T} \Phi u(x+y)\,dx \nonumber \\
  & \geq \frac{1}{2} |u(y)|^{2} + c \int |Du(x+y)|^{2} \frac{dx}{|x|} \nonumber \\
  & \geq \frac{1}{2} |u(y)|^{2} + c \int \frac{|Du(x)|^{2}}{|x-y|}\,dx.
  \label{eqn_Lame_wpd_y}
\end{align}
\end{remark}

\section{Proof of Theorem \ref{thm_Lame_wt}\label{sec_Lame_wt_suf}}
In the next lemma and henceforth, we use the notation
\begin{align*}
  m_{\rho}(u) & = \rho^{-3} \int_{\Omega \cap S_{\rho}} |u(x)|^{2}\,dx,\qquad S_{\rho} = \bigl\{ x: \rho < |x| < 2\rho \bigr\}, \\
  M_{\rho}(u) & = \rho^{-3} \int_{\Omega \cap B_{\rho}} |u(x)|^{2}\,dx.
\end{align*}

\begin{lemma}
Suppose $L$ is positive definite with weight $\Phi$, and let $u = (u_{i})_{i=1}^{3},\ u_{i} \in \mathring{H}^{1}(\Omega)$ be a solution of
\begin{displaymath}
  L u = 0\qquad \text{on } \Omega \cap B_{2\rho}.
\end{displaymath}
Then
\begin{displaymath}
  \int_{\Omega} [L (u \eta_{\rho})]^{T} \Phi_{y} u \eta_{\rho}\,dx \leq c m_{\rho}(u),\qquad \forall y \in B_{\rho},
\end{displaymath}
where $\eta_{\rho}(x) = \eta(x/\rho),\ \eta \in C_{0}^{\infty}(B_{5/3}),\ \eta = 1$ on $B_{4/3}$, and $\Phi_{y}(x) = \Phi(x-y)$.
\label{lmm_Lame_wt_suf_1}
\end{lemma}

\begin{proof}
By definition of $u$,
\begin{displaymath}
  \int_{\Omega} [L (u \eta_{\rho})]^{T} \Phi_{y} u \eta_{\rho}\,dx = \int_{\Omega} [L (u \eta_{\rho})]^{T} \Phi_{y} u \eta_{\rho}\,dx - \int_{\Omega} (L u)^{T} \Phi_{y} u \eta_{\rho}^{2}\,dx,
\end{displaymath}
where the second integral on the right side vanishes and the first one equals
\begin{displaymath}
  -\int_{\Omega} \Bigl[ 2 D_{k} u_{i} D_{k} \eta_{\rho} + u_{i} D_{kk} \eta_{\rho} + \alpha \bigl( D_{i} u_{k} D_{k} \eta_{\rho} + D_{k} u_{k} D_{i} \eta_{\rho} + u_{k} D_{ki} \eta_{\rho} \bigr) \Bigr] u_{j} \eta_{\rho} (\Phi_{y})_{ij}\,dx.
\end{displaymath}
Note that $D\eta_{\rho},\ D^{2} \eta_{\rho}$ have compact support in $R := B_{5\rho/3} \setminus B_{4\rho/3}$, $|D^{k} \eta_{\rho}| \leq c\rho^{-k}$, and
\begin{displaymath}
  |\Phi_{y,ij}(x)| \leq \frac{c}{|x-y|} \leq c\rho^{-1},\qquad \forall x \in R,\ \forall y \in B_{\rho}.
\end{displaymath}
Thus
\begin{align*}
  \lefteqn{\int_{\Omega} [L (u \eta_{\rho})]^{T} \Phi_{y} u \eta_{\rho}\,dx \leq c \int_{\Omega \cap R} \rho^{-2} |u| |Du|\,dx + c \int_{\Omega \cap R} \rho^{-3} |u|^{2}\,dx} \\
  &\qquad \leq c \biggl[ \rho^{-3} \int_{\Omega \cap S_{\rho}} |u|^{2}\,dx \biggr]^{1/2} \biggl[ \rho^{-1} \int_{\Omega \cap R} |Du|^{2}\,dx \biggr]^{1/2} + c \rho^{-3} \int_{\Omega \cap S_{\rho}} |u|^{2}\,dx.
\end{align*}
The lemma then follows from the well-known local energy estimate \cite{Agmon1964}
\begin{displaymath}
  \rho^{-1} \int_{\Omega \cap R} |Du|^{2}\,dx \leq \rho^{-3} \int_{\Omega \cap S_{\rho}} |u|^{2}\,dx.
\end{displaymath}
\end{proof}

Combining \eqref{eqn_Lame_wpd_y} (with $u$ replaced by $u \eta_{\rho}$) and Lemma \ref{lmm_Lame_wt_suf_1}, we arrive at the following local estimate.
\begin{corollary}
Let the conditions of Lemma \ref{lmm_Lame_wt_suf_1} be satisfied. Then
\begin{displaymath}
  |u(y)|^{2} + \int_{\Omega \cap B_{\rho}} \frac{|Du(x)|^{2}}{|x-y|}\,dx \leq c m_{\rho}(u),\qquad \forall y \in \Omega \cap B_{\rho}.
\end{displaymath}
\label{cor_Lame_wt_suf_1}
\end{corollary}

To proceed, we need the following Poincar\'{e}-type inequality (see \cite{Maz'ya1963}).

\begin{lemma}
Let $u = (u_{i})_{i=1}^{3}$ be any vector function with $u_{i} \in \mathring{H}^{1}(\Omega)$. Then for any $\rho > 0$,
\begin{displaymath}
  m_{\rho}(u) \leq \frac{c}{\mathrm{cap}(\bar{S}_{\rho} \setminus \Omega)} \int_{\Omega \cap S_{\rho}} |Du|^{2}\,dx
\end{displaymath}
where $c$ is independent of $\rho$.
\label{lmm_Lame_wt_suf_2}
\end{lemma}

The next corollary is a direct consequence of Corollary \ref{cor_Lame_wt_suf_1} and Lemma \ref{lmm_Lame_wt_suf_2}.
\begin{corollary}
Let the conditions of Lemma \ref{lmm_Lame_wt_suf_1} be satisfied. Then
\begin{displaymath}
  |u(y)|^{2} + \int_{\Omega \cap B_{\rho}} \frac{|Du(x)|^{2}}{|x-y|}\,dx \leq \frac{c}{\mathrm{cap}(\bar{S}_{\rho} \setminus \Omega)} \int_{\Omega \cap S_{\rho}} |Du|^{2}\,dx,\qquad \forall y \in \Omega \cap B_{\rho}.
\end{displaymath}
\label{cor_Lame_wt_suf_2}
\end{corollary}

We are now in a position to prove the following lemma, which is the key ingredient in the proof of Theorem \ref{thm_Lame_wt}.

\begin{lemma}
Suppose $L$ is positive definite with weight $\Phi$, and let $u = (u_{i})_{i=1}^{3},\ u_{i} \in \mathring{H}^{1}(\Omega)$ be a solution of $L u = 0$ on $\Omega \cap B_{2R}$. Then, for all $\rho \in (0,R)$,
\begin{equation}
  \sup_{x \in \Omega \cap B_{\rho}} |u(x)|^{2} + \int_{\Omega \cap B_{\rho}} |Du(x)|^{2} \frac{dx}{|x|} \leq c_{1} M_{2R}(u) \exp \biggl[ -c_{2} \int_{\rho}^{R} \mathrm{cap}(\bar{B}_{r} \setminus \Omega) r^{-2}\,dr \biggr],
  \label{eqn_Lame_wt_suf_est}
\end{equation}
where $c_{1},\ c_{2}$ are independent of $\rho$.
\label{lmm_Lame_wt_suf_3}
\end{lemma}

\begin{proof}
Define
\begin{displaymath}
  \gamma(r) := r^{-1} \mathrm{cap}(\bar{S}_{r} \setminus \Omega).
\end{displaymath}
We first claim that $\gamma(r)$ is bounded from above by some absolute constant $A$. Indeed, The monotonicity of capacity implies that
\begin{displaymath}
  \mathrm{cap}(\bar{S}_{r} \setminus \Omega) \leq \mathrm{cap}(\bar{B}_{r}).
\end{displaymath}
By choosing smooth test functions $\eta_{r}(x) = \eta(x/r)$ with $\eta \in C_{0}^{\infty}(B_{2})$ and $\eta = 1$ on $B_{3/2}$, we also have
\begin{align*}
  \mathrm{cap}(\bar{B}_{r}) & \leq \int_{\dR^{3}} |D\eta_{r}|^{2}\,dx \leq \sup_{x \in \dR^{3}} |D\eta(x)|^{2} \int_{B_{2r}} r^{-2}\,dx \\
  & = \biggl[ \frac{32}{3} \pi \sup_{x \in \dR^{3}} |D\eta(x)|^{2} \biggr] r.
\end{align*}
Hence the claim follows.

We next consider the case $\rho \in (0,R/2]$. Denote the first and second terms on the left side of \eqref{eqn_Lame_wt_suf_est} by $\varphi_{\rho}$ and $\psi_{\rho}$, respectively. From Corollary \ref{cor_Lame_wt_suf_2}, it follows that for $r \leq R$,
\begin{displaymath}
  \varphi_{r} + \psi_{r} \leq \frac{c}{\gamma(r)} (\psi_{2r} - \psi_{r}) \leq \frac{c}{\gamma(r)} (\psi_{2r} - \psi_{r} + \varphi_{2r} - \varphi_{r}),
\end{displaymath}
which implies that
\begin{displaymath}
  \varphi_{r} + \psi_{r} \leq \frac{c}{c + \gamma(r)} (\varphi_{2r} + \psi_{2r}) = \frac{c e^{c_{0} \gamma(r)}}{c + \gamma(r)} \Bigr[ e^{-c_{0} \gamma(r)} (\varphi_{2r} + \psi_{2r}) \Bigr],\qquad \forall c_{0} > 0.
\end{displaymath}
Since $\gamma(r) \leq A$ and
\begin{displaymath}
  \sup_{s \in [0,A]} \frac{c e^{c_{0} s}}{c + s} \leq \max \Bigl\{ 1, \frac{c e^{c_{0} A}}{c + A}, c c_{0} e^{1-c c_{0}} \Bigr\},
\end{displaymath}
it is possible to choose $c_{0} > 0$ sufficiently small so that
\begin{displaymath}
  \sup_{r > 0} \frac{c e^{c_{0} \gamma(r)}}{c + \gamma(r)} \leq 1.
\end{displaymath}
It follows, for $c_{0}$ chosen this way, that
\begin{equation}
  \varphi_{r} + \psi_{r} \leq e^{-c_{0} \gamma(r)} (\varphi_{2r} + \psi_{2r}).
  \label{eqn_Lame_wt_suf_rec}
\end{equation}

By setting $r = 2^{-l} R\ (l \in \mathbb{N})$ and repeatedly applying \eqref{eqn_Lame_wt_suf_rec}, we obtain
\begin{displaymath}
  \varphi_{2^{-l} R} + \psi_{2^{-l} R} \leq \exp \biggl[ -c_{0} \sum_{j=1}^{l} \gamma(2^{-j} R) \biggr] (\varphi_{R} + \psi_{R}).
\end{displaymath}
If $l$ is such that $l \leq \log_{2}(R/\rho) < l+1$, then $\rho \leq 2^{-l} R < 2\rho$ and
\begin{displaymath}
  \varphi_{\rho} + \psi_{\rho} \leq \varphi_{2^{-l} R} + \psi_{2^{-l} R} \leq \exp \biggl[ -c_{0} \sum_{j=1}^{l} \gamma(2^{-j} R) \biggr] (\varphi_{R} + \psi_{R}).
\end{displaymath}
Note that by Corollary \ref{cor_Lame_wt_suf_1},
\begin{displaymath}
  \varphi_{R} + \psi_{R} \leq c m_{R}(u) \leq c M_{2R}(u).
\end{displaymath}
In addition, the subadditivity of the harmonic capacity implies that
\begin{align*}
  \sum_{j=1}^{l} \gamma(2^{-j} R) & \geq \sum_{j=1}^{l} \frac{\mathrm{cap}(\bar{B}_{2^{1-j} R} \setminus \Omega) - \mathrm{cap}(\bar{B}_{2^{-j} R} \setminus \Omega)}{2^{-j} R} \\
  & = \biggl[ \frac{\mathrm{cap}(\bar{B}_{R} \setminus \Omega)}{2^{-1} R} - \frac{\mathrm{cap}(\bar{B}_{2^{-l} R} \setminus \Omega)}{2^{-l} R} \biggr] + \sum_{j=1}^{l-1} \frac{\mathrm{cap}(\bar{B}_{2^{-j} R} \setminus \Omega)}{2^{-j} R} \\
  & = \frac{1}{2} \cdot \frac{\mathrm{cap}(\bar{B}_{R} \setminus \Omega)}{R} - 2\, \frac{\mathrm{cap}(\bar{B}_{2^{-l} R} \setminus \Omega)}{2^{-l} R} + \sum_{j=1}^{l} \frac{\mathrm{cap}(\bar{B}_{2^{-j} R} \setminus \Omega)}{2^{-j} R} \\
  & \geq -2\, \frac{\mathrm{cap}(\bar{B}_{2^{-l} R} \setminus \Omega)}{2^{-l} R} + \frac{1}{2} \sum_{j=0}^{l} \frac{\mathrm{cap}(\bar{B}_{2^{-j} R} \setminus \Omega)}{2^{-j} R}.
\end{align*}
Since
\begin{align*}
  \frac{\mathrm{cap}(\bar{B}_{2^{-l} R} \setminus \Omega)}{2^{-l} R} & \leq A, \\
  \sum_{j=0}^{l} \frac{\mathrm{cap}(\bar{B}_{2^{-j} R} \setminus \Omega)}{2^{-j} R} & \geq \frac{1}{2} \sum_{j=1}^{l+1} \frac{\mathrm{cap}(\bar{B}_{2^{1-j} R} \setminus \Omega)}{(2^{-j} R)^{2}} \cdot 2^{-j} R \\
  & \geq \frac{1}{2} \sum_{j=1}^{l+1} \int_{2^{-j} R}^{2^{1-j} R} \mathrm{cap}(\bar{B}_{r} \setminus \Omega) r^{-2}\,dr \\
  & \geq \frac{1}{2} \int_{\rho}^{R} \mathrm{cap}(\bar{B}_{r} \setminus \Omega) r^{-2}\,dr,
\end{align*}
we have
\begin{displaymath}
  \exp \biggl[ -c_{0} \sum_{j=1}^{l} \gamma(2^{-j} R) \biggr] \leq \exp \biggl[ -\frac{c_{0}}{4} \int_{\rho}^{R} \mathrm{cap}(\bar{B}_{r} \setminus \Omega) r^{-2}\,dr + 2c_{0} A \biggr].
\end{displaymath}
Hence \eqref{eqn_Lame_wt_suf_est} follows with $c_{1} = c e^{2c_{0} A}$ and $c_{2} = c_{0}/4$.

Finally we consider the case $\rho \in (R/2,R)$. By Corollary \ref{cor_Lame_wt_suf_1},
\begin{displaymath}
  |u(y)|^{2} + \int_{\Omega \cap B_{\rho}} \frac{|Du(x)|^{2}}{|x-y|}\,dx \leq c m_{\rho}(u),\qquad \forall y \in \Omega \cap B_{\rho},
\end{displaymath}
which implies that
\begin{displaymath}
  \sup_{y \in \Omega \cap B_{\rho}} |u(y)|^{2} + \int_{\Omega \cap B_{\rho}} |Du(x)|^{2} \frac{dx}{|x|} \leq c M_{2R}(u).
\end{displaymath}
In addition,
\begin{displaymath}
  \int_{\rho}^{R} \mathrm{cap}(\bar{B}_{r} \setminus \Omega) r^{-2}\,dr \leq A \int_{R/2}^{R} r^{-1}\,dr = A \log 2,
\end{displaymath}
so
\begin{displaymath}
  \biggl[ \sup_{y \in \Omega \cap B_{\rho}} |u(y)|^{2} + \int_{\Omega \cap B_{\rho}} |Du(x)|^{2} \frac{dx}{|x|} \biggr] \exp \biggl[ c_{2} \int_{\rho}^{R} \mathrm{cap}(\bar{B}_{r} \setminus \Omega) r^{-2}\,dr \biggr] \leq c_{1} M_{2R}(u)
\end{displaymath}
provided that $c_{1} \geq c e^{c_{2} A \log 2}$.
\end{proof}

\begin{proofof}{Theorem \ref{thm_Lame_wt}}
Consider the Dirichlet problem \eqref{eqn_Dirichlet}
\begin{displaymath}
  Lu = f,\qquad f_{i} \in C_{0}^{\infty}(\Omega),\ u_{i} \in \mathring{H}^{1}(\Omega).
\end{displaymath}
Since $f$ vanishes near the boundary, there exists $R > 0$ such that $f = 0$ in $\Omega \cap B_{2R}$. By Lemma \ref{lmm_Lame_wt_suf_3},
\begin{displaymath}
  \sup_{x \in \Omega \cap B_{\rho}} |u(x)|^{2} \leq c_{1} M_{2R}(u) \exp \biggl[ -c_{2} \int_{\rho}^{R} \mathrm{cap}(\bar{B}_{r} \setminus \Omega) r^{-2}\,dr \biggr],
\end{displaymath}
and in particular,
\begin{displaymath}
  \limsup_{x \to 0} |u(x)|^{2} \leq c_{1} M_{2R}(u) \exp \biggl[ -c_{2} \int_{0}^{R} \mathrm{cap}(\bar{B}_{r} \setminus \Omega) r^{-2}\,dr \biggr] = 0,
\end{displaymath}
where the last equation follows from the divergence of the Wiener integral
\begin{displaymath}
  \int_{0}^{1} \mathrm{cap}(\bar{B}_{r} \setminus \Omega) r^{-2}\,dr = \infty.
\end{displaymath}
Hence $O$ is regular with respect to $L$.
\end{proofof}

\end{document}